\documentclass[final,3p,12pt]{elsarticle}
\usepackage{amsmath,amssymb,amsfonts,verbatim}
\journal{Advances in Mathematics}

\newtheorem{theorem}{Theorem}
\newtheorem{proposition}{Proposition}
\newdefinition{remark}{Remark}
\newproof{proof}{Proof}

\def\a{\a}
\def\ad{{\rm ad}}
\def\Ad{{\rm Ad}}
\def\al{\alpha}

\def\CC{\mathbb{C}}
\def\com{\ts,\hskip-.5pt}

\def\De{\Delta}

\def\f{\varphi}
\def\ff{\mathfrak{f}}

\def\g{\mathfrak{g}}
\def\Gr{\operatorname{G}}
\def\ga{\gamma}

\def\gp{\g^{\ts\prime}}
\def\gr{\operatorname{gr}}

\def\la{\lambda}
\def\lcd{,\ldots,}

\def\Mb{\,\overline{\!M}}

\def\n{{\mathfrak{n}_+}}
\def\nminus{{\mathfrak{n}_-}}
\def\Nr{\operatorname{Norm}}

\def\op{\oplus}
\def\ot{\otimes}

\def\S{\operatorname{S}}
\def\sih{\tilde{s}}
\def\sl{\mathfrak{sl}}

\def\t{\mathfrak{t}}
\def\Tr{\operatorname{T}}
\def\ts{\hskip1pt}

\def\U{\operatorname{U}}

\def\Y{\operatorname{Y}}

\def\Zb{\,\overline{\!Z}}
\def\ZZ{{\mathbb Z}}

\begin{document}

%-----------------------------------------------------------------------------

\begin{frontmatter}

\title{A generalized Harish-Chandra isomorphism}
 
\author[a,b]{Sergey Khoroshkin}
\address[a]{Institute for Theoretical and Experimental Physics, 
Moscow 117259, Russia}
\address[b]{Department of Mathematics, Higher School of Economics,
Moscow 117312, Russia}
%\ead{khor@itep.ru}

\author[c]{Maxim Nazarov}
\address[c]{Department of Mathematics,  University of York, 
York YO10 5DD, England}
%\ead{mln1@york.ac.uk}

\author[d]{Ernest Vinberg}
\address[d]{Department of Mathematics, Moscow State University, 
Moscow 119992, Russia}
%\ead{vinberg@zebra.ru}

\begin{keyword}
Chevalley theorem\sep
Harish-Chandra isomorphism\sep
Zhelobenko operator
%\MSC[2000]17B10
\end{keyword}

\begin{abstract}
For any complex reductive Lie algebra $\g$ and any
locally finite $\g\ts$-module $V$, we extend to the tensor product 
$\U(\g)\ot V$ the Harish-Chandra description of $\g\ts$-invariants
in the universal enveloping algebra $\U(\g)\ts$.
\end{abstract}

\end{frontmatter}

\thispagestyle{empty}%%%%%%%%%%%%%%%%%%%%%%%%%%%%%%%%%%%%%%%%%%%%%%%%%%%%%%%%%

%=============================================================================

\section*{Introduction}

Let $\g$ be a complex semisimple Lie algebra with a given decomposition 
$\g=\nminus\op\t\op\n$ to a Cartan subalgebra $\t$ and the nilpotent
radicals $\n$ and $\nminus $ of two opposite Borel subalgebras
containing $\t\ts$. 
Take the projection of the universal enveloping algebra $\U(\g)$ to $\U(\t)$ 
parallel to $\nminus\!\U(\g)+\U(\g)\,\n\ts$.  
Harish-Chandra proved \cite{HC} that the restriction of the projection
map to the center $\mathrm{Z}(\g)$ of $\U(\g)$ establishes an
isomorphism of $\mathrm{Z}(\g)$ with the subalgebra of
invariants in $\U(\t)$
relative to the shifted action of the Weyl group $W$ of $\g\ts$. 
In the graded commutative setting, the Harish-Chandra theorem
reduces to the Chevalley theorem, which states 
that restriction from $\g$ to $\t$ yields an isomorphism
between the algebra of $\g\ts$-invariant polynomial functions on $\g$
and the algebra of $W$-invariant polynomial functions on~$\t\ts$.

Now let $V$ be any $\g\ts$-module which is a direct sum of
its \text{irreducible} finite-dimensional submodules. 
Take the tensor product $M=\U(\g)\ot V$ of 
$\g\ts$-modules,
and consider its subspace of invariants $M^{\ts\g}\ts$.
We provide an explicit description of %$M^{\ts\g}\ts$,
this subspace, thus
generalizing the Harish-Chandra theorem.
Namely, the vector space $M$ has a natural $\U(\g)\ts$-bimodule
structure, see the definitions (\ref{lact}),(\ref{ract}).
The diagonal action of $\g$ on $M$ coincides with the action 
adjoint to the $\U(\g)\ts$-bimodule structure. Consider
the quotient vector space
$
Z=M/(\nminus  M+M\n)
$
and the %corresponding 
projection map $M\to Z\ts$.
This map turns out to be injective on the subspace
$M^{\ts\g}\subset M\ts$, and we describe the image 
of this subspace in $Z\ts$.

Let $r$ be the rank of $\g\ts$, and
let $\De\subset\t^*$ be the root system of $\g\ts$. 
Let $\al_1\lcd\al_r\in\De$ be simple roots, and
let $s_1\lcd s_r\in W$ be the corresponding simple reflections.
For each $i=1\lcd r$ let $H_i=\al_i^\vee\in\t$ be the coroot 
corresponding to $\al_i\ts$. Let $E_i\in\n$ and
$F_i\in\nminus$  be root vectors corresponding to the roots
$\al_i$ and $-\al_i\ts$. Then let $\g_i$ be the $\sl_2\ts$-subalgebra
spanned by $E_i,F_i,H_i\ts$. For each $j=0,1,2,\ldots$
let $V_{ij}$ be the sum of $(2j+1)$-dimensional irreducible
$\g_i\ts$-submodules of $V$.
For any given $i$ the zero weight subspace 
$V^{\ts0}\subset V$ relative to $\t$
is a direct sum over $j=0,1,2,\ldots$
of the intersections $V^{\ts0}\cap V_{ij}\ts$.

The quotient 
vector space $Z$ can be identified with the
tensor product $\U(\t)\ot V$.
We prove that an element of $Z=\U(\t)\ot V$
lies in the image of $M^{\ts\g}$
if and only if for each $i=1\lcd r$ it can be written as
a finite sum of products of the form
$\Theta\,\Psi_{ij}\ot v$ where $v\in V^{\ts0}\cap V_{ij}$,
$\Theta\in\U(\t)$ is invariant under the shifted action of %the reflection
$s_i\in W$, and
\begin{equation*}
\label{psij}
\Psi_{ij}=(H_i+2)(H_i+3)\ldots(H_i+j+1)\ts.
\end{equation*}

This description of the image of $M^{\ts\g}$ in %the space
$Z=\U(\t)\ot V$
has a graded commutative version, 
which generalizes the Chevalley restriction theorem.
Namely, let $\S(\g)$ be the symmetric algebra of $\g\ts$.
Take the tensor product $\S(\g)\ot V$ of $\g\ts$-modules.
%and its subspace of $\g\ts$-invariants. 
The projection map 
$$
\S(\g)\to\S(\g)/(\nminus+\n)\S(\g)=\S(\t)
$$
yields a map $\S(\g)\ot V\to\S(\t)\ot V$,
which is injective on the subspace of $\g\ts$-invariants.
We prove that the image of this subspace in $\S(\t)\ot V$
consists of all $W$-invariant elements
$F\in\S(\t)\ot V^{\ts0}$ satisfying 
the following property: 
for every $i=1\lcd r$ the projection of $F$ to each
$\S(\t)\otimes (V^0\cap V_{ij})$ with $j=0,1,2,\ldots$  
is divisible by $H_i^j$ in
the first tensor factor.
Here we use the standard actions of the Weyl group 
$W$ on $\S(\t)$ and on $V^{\ts0}$.
A criterion on $V$, for all
$W$-invariant elements of $\S(\t)\ot V^{\ts0}$
to be images of $\g\ts$-invariant elements of $\S(\g)\ot V$,
was given by Broer \cite{B}.
The sufficiency of Broer's condition also follows from our 
result,~see~Remark~\ref{R3}.

We also give another description of the image
of $M^{\ts\g}$ in $Z\ts$, which
does not need splitting $V$ into the sum
of its $\g_i\ts$-isotypical \text{components.}
Let $\overline{\U(\t)\!}\,$ be the ring of fractions of the 
commutative algebra $\U(\t)$ relative to the set of denominators
\eqref{M2}. Consider the left $\overline{\U(\t)\!}\,$-module
$$
\Zb={\overline{\U(\t)\!}\,}\ot_{\,\U(\t)}Z
$$
where we use the left action of $\U(\t)$ on $Z\ts$.
If the quotient vector space $Z$ is identified with
$\U(\t)\ot V$, then $\Zb$ is identified with 
${\overline{\U(\t)\!}\,}\ot V$.
The right action of $\U(\t)$ on $Z$ also extends to
a right action of ${\overline{\U(\t)\!}\,}$~on~$\Zb\ts$.

We use certain linear operators $\xi_1,\xi_2,...,\xi_r$
on $\Zb\ts$, which we call the Zhelobenko operators.
They originate from the \textit{extremal cocycle\/}
on the Weyl group $W$, defined in \cite{Z1}.
Using this cocycle, for any pair
$(\mathfrak{f}\com\g)$
with a finite-dimensional Lie algebra $\mathfrak{f}$
containing $\g\ts$, Zhelobenko
constructed a resolution of the subspace $N^{\ts\n}$ of 
$\n\ts$-invariants of any $\mathfrak{f}\ts$-module $N$.
Here $N$ is regarded as a $\g\ts$-module~by~restriction.
In their present form, the operators $\xi_1,\xi_2,...,\xi_r$
have been defined by Khoroshkin and Ogievetsky \cite{KO}
as automorphisms of Mickelsson algebras \cite{Mi}.
They satisfy the braid relations corresponding to 
$\g\ts$. This braid group action is closely 
related to the \textit{dynamical Weyl group action}
due to Etingof, Tarasov and Varchenko \cite{EV,TV}.

The Zhelobenko operators $\xi_1\lcd\xi_r$ on $\Zb$
preserve the zero weight subspace 
$\Zb{}^{\ts0}$ 
relative to the adjoint action of $\t$ on $\Zb\ts$,
and moreover are involutive on this subspace. 
So we get an action of the Weyl group $W$ on $\Zb{}^{\ts0}$,
such that each 
$s_i\in W$ acts as the operator $\xi_i\ts$.
We prove that 
the image of $M^{\ts\g}$ in $Z$ consists
of all elements of $Z\cap\Zb{}^{\ts0}$
which are invariant under the latter action of $W$.

This result is remarkably similar in spirit to those of Kostant 
and Tirao \cite{KT}, who studied the subalgebra 
$\U(\ff)^{\ts\g}\subset\U(\ff)$ where $\ff$ and $\g$ 
are the complexified Lie algebras
of a real connected semisimple Lie group and of its maximal compact
subgroup respectively. In~\cite{KT} the subalgebra of $\g\ts$-invariants
$\U(\ff)^{\ts\g}$ 
was mapped injectively to the tensor product $\U(\g)\ot\U(\mathfrak{a})$ 
where $\mathfrak{a}$ is the Cartan subalgebra of the symmetric pair
$(\ff\com\g)\ts$. To describe the image of $\U(\ff)^{\ts\g}$ under this map, 
a version of the intertwining operators of Knapp and Stein
\cite{KS} was used in \cite{KT}, instead of the shifted action of 
the Weyl group of %the pair
$(\ff\com\g)$. A certain localization of the ring $\U(\g)\ot\U(\mathfrak{a})$
was also used in \cite{KT}. This result
has been generalized by Oda \cite{O}.

We will work in a setting slightly more general
than used in the beginning of this section. 
%Starting from the next section, 
Our $\g$ will be any reductive complex Lie algebra. 
%of semisimple rank $r\ts$. 
Then we will fix a connected reductive complex 
algebraic group $\Gr$ with the Lie algebra $\g\ts$.
Our $V$ will be any $\Gr\ts$-module which 
can be decomposed
into a direct sum of irreducible finite-dimensional $\Gr\ts$-submodules.
We regard $M=\U(\g)\ot V$ as a  
$\Gr\ts$-module, and study its subspace %$M^{\ts\Gr}$ 
of $\Gr\ts$-invariants. Our Theorem~1 and 
Proposition~4 describe this subspace explicitly.
Theorem~2 is a graded commutative version of this description.
We give a geometric proof of Theorem~2. 
Then we prove Theorem~1, and extend it to 
disconnected algebraic~groups.

Our work is motivated by the publication \cite{KN1},
which was inspired
by the results of Tarasov and Varchenko \cite{TV}.
The publication \cite{KN1}
established a correspondence between Zhelobenko automorphisms
of certain Mickelsson algebras, and canonical
intertwining operators 
of tensor products of the fundamental 
of representations of Yangians. 
Namely, these are
the Yangian $\Y(\mathfrak{gl}_n)$ of the general linear Lie algebra,
and its twisted analogues $\Y(\mathfrak{sp}_n)$ and $\Y(\mathfrak{so}_n)$
which correspond to the symplectic and orthogonal Lie algebras.
%The Yangian $\Y(\mathfrak{gl}_n)$ is a Hopf algebra,
%while $\Y(\mathfrak{sp}_n)$ and $\Y(\mathfrak{so}_n)$
%are one-sided coideal subalgebras of $\Y(\mathfrak{gl}_n)\ts$.
For an introduction to the theory of 
Yangians see the recent book by Molev \cite{Mo}.

%Each of these Yangians contains as a subalgebra
%the enveloping algebra of the corresponding
%classical Lie algebra.
In a forthcoming publication \cite{KN5} of the first two authors
of the present article, Theorem~1 is used to
solve a basic problem in the representation theory of Yangians.
Up to the action of the centre of $\Y(\mathfrak{sp}_n)\ts$, 
every irreducible finite-dimensional
$\Y(\mathfrak{sp}_n)\ts$-module is realized
as the image of an intertwining operator of tensor products
of fundamental representations of $\Y(\mathfrak{sp}_n)\ts$.
For the Yangian $\Y(\mathfrak{gl}_{\ts n})$
such realizations were provided by
Akasaka and Kashiwara \cite{AK}, by Cherednik \cite{C},
and by Nazarov and Tarasov \cite{NT}.
Our Theorem~1 yields new proofs of these results for 
$\Y(\mathfrak{gl}_{\ts n})\ts$.
For the twisted Yangian $\Y(\mathfrak{so}_n)\ts$, the images
of our intertwining operators realize,
up to the action of the centre of the algebra 
$\Y(\mathfrak{so}_n)\ts$,
all those irreducible finite-dimensional representations,
where the action of the subalgebra 
$\U(\mathfrak{so}_n)$ of $\Y(\mathfrak{so}_n)$
integrates to an action of the complex 
special~orthogonal~group~$\mathrm{SO}_n\ts$.

%=============================================================================

\section{Notation}

Let $\g$ be any reductive complex Lie algebra of semisimple rank $r$.
Choose a \textit{triangular decomposition\/}
$$
\g=\nminus\op\t\op\n
$$
where $\t$ is a Cartan subalgebra, while $\n$ and $\nminus$ are 
the nilpotent radicals of two opposite Borel subalgebras of $\g$ 
containing $\t\ts$. Let $\De\subset\t^*$ be the root system of $\g\ts$. 
%and $\De^+\subset\De$ be the subset of positive roots. 
Let $\al_1\lcd\al_r\in\De$ be simple roots.

For each $i=1\lcd r$ let $H_i=\al_i^\vee\in\t$ be the coroot 
corresponding to the simple root $\al_i\ts$. Let $E_i\in\n$ and
$F_i\in\nminus$ be root vectors corresponding to the roots
$\al_i$ and $-\al_i\ts$. We assume that $[E_i\com F_i]=H_i\ts$.

Let $W$ be the Weyl group of the root system $\De\ts$. Let
$s_1\lcd s_r\in W$ be the reflections corresponding to the 
simple roots $\al_1\lcd\al_r\ts$. 
Let $\rho$ be the half-sum 
of the positive roots. Then the \textit{shifted action\/} 
$\circ$ of the group $W$ on the vector space $\t^*$ is defined by setting
\begin{equation*}
w\circ\la=w(\la+\rho)-\rho.
\end{equation*}
In particular,
$$
s_i\circ\la=s_i(\la+\al_i).
$$
The action $\circ$ extends to an action of $W$ by automorphisms 
of the symmetric algebra $\S(\t)=\U(\t)\ts$, 
by regarding elements of this algebra as polynomial functions on $\t^*$:
$$
(w\circ\Phi)(\la)=\Phi(w^{-1}\circ\lambda)
\quad\text{for}\quad
\Phi\in \U(\t)\ts.\
$$
Note that $s_i\circ H_i=-H_i-2\ts$. 
It follows that for the elements
$\Psi_{ij}$ defined in the Introduction %we have
\begin{equation}
\label{newformula}
s_i\circ\Psi_{ij}=(-1)^j\,H_i(H_i-1)...(H_i-j+1)\ts.
\end{equation}

For any $\g\ts$-module $M$ 
we denote by $M^{\ts0}$ its zero weight 
subspace relative to $\t\ts$.

Let $\Gr$ be a connected reductive complex algebraic group with 
Lie algebra $\g$. Let $V$ be a locally finite $\Gr\ts$-module. 
This means that $V$ can be decomposed into a direct sum of 
irreducible finite-dimensional $\Gr\ts$-submodules.
Put $M=\U(\g)\ot V\ts$.
Since the adjoint action of the group $\Gr$ on $\U(\g)$ is 
locally finite, 
the same holds for the diagonal action of $\Gr$ on $M$. 
We will also consider the action of the Lie algebra $\g$ on 
$M$ corresponding to that of %the group 
$\Gr\ts$.

We will also regard $M$ as a $\U(\g)\ts$-bimodule by using 
the left and right actions defined for $X\in\g$ by the formulas
\begin{align}
\label{lact}
X\ts(u\ot v)&=(Xu)\ot v,
\\
\label{ract}
(u\ot v)\ts X&=(u\ts X)\ot v-u\ot(Xv)
\end{align}
where $u\in\U(\g)\ts$, $v\in V\ts$.
Then the initial (diagonal) action of $\g$ on $M$ coincides with the action 
$\ad$ \textit{adjoint} to the $\U(\g)\ts$-bimodule structure:
$$
(\ad\ts X)\ts m=X\ts m-m\ts X 
\quad\text{for}\quad 
X\in\g
\quad\text{and}\quad 
m\in M\ts.
$$

Note that $M$ is a free left and a free right $\U(\g)\ts$-module. 
Further, $M$ admits the following decomposition into a direct sum of 
$\U(\t)\ts$-bimodules:
\begin{equation}
\label{MT}
M=\U(\t)\ot V\op(\nminus M+M\n)\ts,
\end{equation}
\noindent
see \cite[Proposition 3.3]{KO}. Set
$
Z=M/(\nminus  M+M\n)\ts.
$
By \eqref{MT}, the restriction of the projection $M\to Z$
to the subspace $\U(\t)\ot V\subset M$ provides a bijection
$$
\U(\t)\ot V\to Z\ts.
$$
%Identifying the space $V\subset M$ with its image 
%under the canonical map $M\to Z$, one can write
%$$
%$Z=\U(\t)V=V\U(\t).
%$$
Moreover, $Z$ is a free left and a free right $\U(\t)\ts$-module.
Let $Z^{\ts0}\subset Z$ be the zero weight subspace 
relative to the adjoint action of $\t\ts$.

For each root $\al\in\De$ let $H_\al=\al^\vee\in\t$ be the
corresponding coroot. 
Denote by $\overline{\U(\t)\!}\,$ the ring of fractions of the 
commutative algebra $\U(\t)$ relative to the set of denominators
\begin{equation}
\label{M2}
\{\,H_\al+k\ |\ \al\in\Delta,\ k\in\ZZ\,\}\,.
\end{equation}
If the elements of $\U(\t)$ are regarded as polynomial functions on $\t^*$, 
then 
the elements of $\overline{\U(\t)\!}\,$ 
can be regarded as rational functions on $\t^*$. 
The shifted action $\circ$ of the Weyl group $W$ on $\U(\t)$
%naturally 
extends to~$\overline{\U(\t)\!}\,$.

Let $\overline{\U(\g)\!}\,$ be
the ring of fractions of $\U(\g)$ relative to
the set of denominators \eqref{M2}. Put
$$
\Mb=\overline{\U(\g)\!}\,\ot V.
$$
The left action of $\U(\g)$ on $M$ extends to an action of 
$\overline{\U(\g)\!}\,$ on $\Mb$ in a natural way, 
via left multiplication in the first tensor factor of $\Mb\ts$. 

Since $M$ is a locally finite $\g\ts$-module, it is spanned by its 
weight vectors relative to $\t\ts$. 
For any weight vector $m\in M$ and any $\al\in\De$ there is $l\in\ZZ$ 
such that
$m\ts H_\alpha=(H_\alpha+l)\ts m$. Then by setting
$$
m\ts (H_\alpha+k)^{-1}=(H_\alpha+k+l\ts)^{-1} m
$$
for each $k\in\ZZ$, we extend the right action of $\U(\g)$ on $M$
to a right action of $\overline{\U(\g)\!}\,$ on $\Mb$. 
Thus the space $\Mb$ becomes an 
$\overline{\U(\g)\!}\,$-bimodule. It is obviously free 
as a left and as a right $\overline{\U(\g)\!}\,$-module.

The decomposition \eqref{MT} extends to the following decomposition 
of $\Mb$ into a direct sum of $\overline{\U(\t)\!}\,$-bimodules:
\begin{equation}
\label{MbarT}
\Mb=\overline{\U(\t)\!}\,\ot V
\op
(\nminus\Mb+\Mb\ts\n)\ts.
\end{equation}
Set
$$
\Zb=\Mb/
(\nminus\Mb+\Mb\ts\n)\ts.
$$
Note that 
$$
(\nminus\Mb+\Mb\n)\cap M=\nminus M+M\n\ts.
$$
Thus we have a natural embedding $Z\to\Zb\ts$.
Due to \eqref{MbarT}, the 
restriction of the canonical map $\Mb\to\Zb$
to the subspace $\,\overline{\U(\t)\!}\,\ot V\subset\Mb$
provides a bijection 
$$
\,\overline{\U(\t)\!}\,\ot V\to\Zb\ts.
$$
Using this bijection, the above embedding
$Z\to\Zb$ corresponds
to the natural embedding $\U(\t)\to\,\overline{\U(\t)\!}\,$.
Moreover, $\Zb$ is a free left and a free right 
$\,\overline{\U(\t)\!}\,$-module.
Let $\Zb{}^{\ts0}\subset\Zb$ be the zero weight subspace 
relative to the adjoint action of $\t\ts$.

Let $\Tr$ be the maximal torus of the group $\Gr$ 
with the Lie algebra $\t\ts$. Let
$\Nr(\Tr)$ be the normalizer of $\Tr$ in $\Gr\ts$. 
The adjoint action of the group $\Nr(\Tr)$ on $\t$ 
establishes an isomorphism
$$
\Nr(\Tr)/\Tr\to W.
$$
Since this action preserves the set of coroots, 
the induced action of $\Nr(\Tr)$ on $\U(\t)$ extends to its action on 
$\overline{\U(\t)\!}\,$, so that for any $w\in W$
$$
(w\,\Phi)(\la)=\Phi\ts(w^{-1}(\la)),
$$
when the element $\Phi\in\overline{\U(\t)\!}\,$ 
is regarded as a rational function on $\t^*$. 
The adjoint action of $\Nr(\Tr)$ on $\U(\g)$ then extends 
to its action by automorphisms of $\overline{\U(\g)\!}\,$, 
and the adjoint action of $\Nr(\Tr)$ on $M$ extends to a 
$\overline{\U(\g)\!}\,$-bimodule equivariant action of $\Nr(\Tr)$ on 
$\Mb$.

For each $i=1\lcd r$ let $\g_i$ denote the $\sl_2\ts$-subalgebra of 
$\g$ spanned by the elements $E_i,F_i,H_i\ts$. 
Let $\Gr_i$ be the corresponding connected subgroup of $\Gr\ts$. 
Choose a representative of $s_i$ in $\Nr(\Tr)$ lying 
in $\Gr_i$ and denote it by $\sih_i\ts$. %Then
The elements $\sih_1\lcd\sih_r\in\Nr(\Tr)$
satisfy the braid relations
$$
\underbrace{\sih_i\,\sih_j\,\sih_i\,\ldots}_{m_{ij}}\,=\,
\underbrace{\sih_j\,\sih_i\,\sih_j\,\ldots}_{m_{ij}}\, 
\quad\text{for}\quad i\neq j\ts,
$$
where $m_{ij}$ is the order of the element
$s_i\ts s_j$ in the group $W$, see \cite{T}.

%=============================================================================

\section{Zhelobenko operators}

For $i=1\lcd r$ define a %$T$-equivariant 
linear map $\eta_{i}:M\to\Mb$ 
by setting $\eta_i\ts(m)$ for any $m\in M$ to be %the sum
\begin{equation}
\label{M10}
\sum_{k=0}^\infty\,\,
\bigl(k!\,H_i(H_i-1)\ldots(H_i-k+1)\bigr)^{-1}
E_i^k\,(\ad\ts F_i)^k\,\sih_i(m)\ts.
\end{equation}
%Here we assume that $[E_i,F_i]=H_i\ts$.
Since the adjoint action of $\g$ on $M$ is locally finite, 
for any given $m\in M$ only finitely many terms of the sum 
\eqref{M10} may differ from zero. Hence the map $\eta_i$ is well defined. 
The definition \eqref{M10} and the next two propositions go back to
\cite[Section~2]{Z1}. 
See \cite[Section 3]{KN1} for detailed proofs of these two propositions.

\begin{proposition} 
We have $\eta_i\,(\nminus M+M\n)
\subset
(\nminus\Mb+\Mb\n)$.
\end{proposition}

Due to this proposition, the map $\eta_i$ induces a linear map
$
\xi_i:Z\to\Zb\ts.
$

\begin{proposition} 
For any $\Phi\in\U(\t)$ and $z\in Z$ we have
$$
\xi_i\ts(\Phi\ts z)=(s_i\circ\Phi)\,\xi_i(z)\ts.
$$
\end{proposition}

The latter proposition allows us to extend $\xi_i$ to 
$\Zb$ by setting
\begin{equation}
\label{antilin}
\xi_i(\Phi\ts z)=(s_i\circ\Phi)\,\xi_i(z)
\quad\text{for}\quad
\Phi\in\overline{\U(\t)\!}\,
\quad\text{and}\quad 
z\in Z\ts.
\end{equation}

In their present form, the operators $\xi_i$ were introduced
in \cite{KO}. We call them the \textit{Zhelobenko operators}.
The next proposition states the key property of these
operators. For its proof see \cite[Section 6]{Z1}.
Note that our notation differs from that used in \cite{KN1,KO,Z1}.

\begin{proposition}
The operators $\xi_1\lcd\xi_r$ on $\Zb$ satisfy
the braid relations
$$
\underbrace{\xi_i\,\xi_j\,\xi_i\,\ldots}_{m_{ij}}\,=\,
\underbrace{\xi_j\,\xi_i\,\xi_j\,\ldots}_{m_{ij}}\, 
\quad\text{for}\quad i\neq j\ts.
$$
\end{proposition}

The squares of the Zhelobenko operators are given by the formula
$$
\xi_i^2(z)=(H_i+1)\,\sih_i^2(z)\,(H_i+1)^{-1}
\quad\text{for}\quad
z\in\Zb\ts,
$$
see \cite[Corollary 7.5]{KO}. Since $\sih_i^2\in\Tr\ts$, 
the squares $\sih_i^2$ 
and hence $\xi_i^2$ act trivially on the zero weight subspace
$\Zb{}^{\,0}\subset\Zb\ts$. 
This means that the restrictions of operators $\xi_1\lcd\xi_r$ 
to $\Zb{}^{\,0}$ define an action of the Weyl group $W$.
Note that, when applying the operator $\xi_i$ to the~coset 
$$
z=m+(\nminus M+M\n)\in Z
$$
of any $m\in M\ts$,
one can replace $E_i$ by $\ad\ts E_i$ in \eqref{M10}. That is,
modulo $\nminus\ts\Mb+\Mb\ts\n\ts$, the sum \eqref{M10} equals
$$
\sum_{k=0}^\infty\,\,
\bigl(k!\,H_i(H_i-1)\ldots(H_i-k+1)\bigr)^{-1}
(\ad\ts E_i)^k\,(\ad\ts F_i)^k\,\sih_i(m)\ts.
$$

%=============================================================================

\section{Harish-Chandra isomorphism}

By regarding $Z^{\ts0}\subset Z$ as a subspace of $\Zb$, 
let $Q$ be the subspace
of all elements of $Z^{\ts0}$ invariant under all the Zhelobenko operators :
$$
Q=\{\,z\in Z^{\ts0}\ |\ \xi_i(z)=z\ \,\text{for}\ \,i=1\lcd r\}\,.
$$

Now consider the subspace $M^{\ts\Gr}$ of $\Gr\ts$-invariants in $M\ts$. 
Define the linear map $\ga:M^{\ts\Gr}\to Z$ as the restriction to 
$M^{\ts\Gr}$ of the canonical projection $M\to Z\ts$.
It immediately follows from
the definition of the maps $\eta_i$ that 
$\eta_i(m)=m$ for any $m\in M^{\ts\Gr}$.
Hence $\gamma(M^{\ts\Gr})\subset Q\ts$.
The following theorem is the main result of this paper.

\begin{theorem}
The map $\ga$ is injective and its image $\gamma(M^{\ts\Gr})$ equals\/ $Q$.
\end{theorem}

A proof of this theorem will be given in Section~6.

%In the particular case, 
When $V=\CC$ and $M=\U(\g)$, our theorem reproduces the classical 
description of the centre of the universal enveloping algebra $\U(\g)$ 
due to Harish-Chandra \cite{HC}. 
In the general case, we will call the map $\ga:M^{\ts\Gr}\to Q$ 
the \textit{Harish-Chandra isomorphism\/} for $M=\U(\g)\ot V$.

\begin{remark}
To justify this terminology further, let us consider the
special case of our general setting, when $M$ is also
an associative %unital 
algebra which contains
$\U(\g)$ as a subalgebra, such that the $\U(\g)\ts$-bimodule
structure %\eqref{lact},\eqref{ract} 
on $M$ comes from the multiplication in $M\ts$.
As an algebra, $M$
is then generated by its subspaces $\g$ and $1\ot V$,
and we have the commutation relations
$$
[X\com 1\ot v]=1\ot(Xv)
\quad\text{for}\quad
X\in\g
\quad\text{and}\quad
v\in V\ts.
$$

The group $\Gr$ then acts on $M$ by algebra automophisms,
and the subspace $M^{\ts\Gr}\subset M$ is a subalgebra. 
The vector subspaces $\nminus M$ and $M\n$ of $M$ now become
the right and left ideals generated by $\nminus$ and $\n$ 
respectively. Multiplication of the cosets in $Z$
of elements of $M^{\ts\Gr}$ 
by using the algebra structure on $M$ is now well defined.
The map $\ga:M^{\ts\Gr}\to Q$
becomes an isomorphism of algebras.

Furthermore, one can equip the vector space $\Zb$ 
with a natural structure of an associative algebra, 
such that the subspace $Q\subset\Zb$ with the above 
defined multiplication is a subalgebra. Namely, 
there is a certain completion $\,\widetilde{\!\U(\g)\!}\,$ 
of the algebra $\overline{\U(\g)\!}\,$ and a unique 
element $P$ of this completion such that
\begin{gather*}
P^2=P,\quad\n P=P\nminus=0\ts,
\\
P\in(1+\nminus\,\widetilde{\!\U(\g)\!}\,)
\cap(1+\,\widetilde{\!\U(\g)\!}\,\n)\ts.
\end{gather*}
The element $P$ %$\in\,\widetilde{\!\U(\g)\!}\,$ 
is called the \textit{extremal projector\/} for the Lie algebra $\g\ts$; 
its definition is due to 
R.\,Asherova, Y.\,Smirnov and V.\,Tolstoy \cite{AST}. 
Multiplication in $\Zb$ is defined by setting the product of 
the cosets of two elements $m\com n\in\Mb$ to be the coset 
of $m\ts P\ts n\ts$. 
Then $\Zb$ is called the \textit{Mickelsson algebra}, 
see \cite[Section~3]{KO} for details of this definition, 
and for links of $\Zb$ to other Mickelsson algebras.
Each %operator 
$\xi_i$ is an automorphism of the algebra $\Zb$ 
by \cite[Section~5]{KO}. The projection $\ga:M^{\ts\Gr}\to\Zb$ 
becomes an injective homomorphism of algebras with the image equal to~$Q$.
\end{remark}

%=============================================================================

\section{Description of the space\ $Q$}

In this section, we identify the vector space $V$ with
the image of the subspace $1\ot V\subset\Mb$ under the
canonical projection $\Mb\to\Zb\ts$. Then we can write 
$Z=\U(\t)\ts V$ and $Z^{\ts0}=\U(\t)\ts V^{\ts0}\ts$.
If we decompose $V$ into a sum of irreducible $\g_i\ts$-submodules, 
then the zero weight subspace $V^{\ts0}$ will lie in the sum
of odd-dimensional summands of $V$. For each $j=0,1,2,\ldots$ 
denote by $V_{ij}$ the sum of $(2j+1)$-dimensional irreducible 
$\g_i\ts$-submodules of $V$.
It follows from \eqref{antilin} and the observation
made at the very end of Section 2, that
$$
\xi_i(\ts\U(\t)\ts V_{ij})\subset\overline{\U(\t)\!}\,\ts V_{ij}\ts.
$$
Hence for each $i=1\lcd r$
the subspace $Q\subset\U(\t)\ts V^{\ts0}$ is the sum of 
its intersections with the subspaces
$\U(\t)\ts(V^{\ts0}\cap V_{ij})\ts$.

Take any vector $v\in V^0\cap V_{ij}\ts$. Then $\sih_i(v)=(-1)^j v$ and
for each $k=0,1\lcd j$ %we have
$$
(\ad\ts E_i)^k\,(\ad\ts F_i)^k\,v=
(j-k+1)(j-k+2)\ldots(j+k-1)(j+k)\,v\ts,
$$
while for $k>j$ we have
$$
(\ad\ts E_i)^k\,(\ad\ts F_i)^k\,v=0\ts.
$$
Hence
$$
\xi_i(v)=(-1)^j\biggl(\,\sum_{k=0}^j\,
\frac{(j-k+1)(j-k+2)\ldots(j+k-1)(j+k)}
{k!\,H_i(H_i-1)\ldots(H_i-s+1)}\biggr)\ts v\ts.
$$
The sum in brackets is a particular value 
$\mathrm{F}\ts(-j\com j+1\com -H_i\ts;1)$
of the hypergeometric function $\mathrm{F}\ts$. 
By the Gauss formula
$$
\mathrm{F}(a,b,c\ts;1)=
\frac
{\,\mathrm\Gamma(c)\ts\mathrm\Gamma(c-a-b)}
{\,\mathrm\Gamma(c-a)\ts\mathrm\Gamma(c-b)}
$$
valid for $a,b,c\in\CC$ with $c\neq 0,-1,\ldots$ 
and $\mathrm{Re}\ts(c-a-b)>0$,~we~get
\begin{align*}
\nonumber
\xi_i(v)&=(-1)^j
\frac
{\mathrm\Gamma(-H_i)\ts\mathrm\Gamma(-H_i-1)}
{\mathrm\Gamma(-H_i+j)\ts\mathrm\Gamma(-H_i-j-1)}\,v 
\\[4pt]
&=(-1)^j
\frac{\,(H_i+2)(H_i+3)\ldots(H_i+j+1)}{\,H_i(H_i-1)\ldots(H_i-j+1)}\,v
=
\frac{\,\Psi_{ij}}{s_i\circ\Psi_{ij}}\,v
\end{align*}
where $\Psi_{ij}$ has been defined in the Introduction. %to this article.
%by the formula \eqref{psij}.
It now follows from Proposition~2 that 
for any $\Phi\in\U(\t)$ we have
$$
\xi_i\ts(\Phi\ts v)=(s_i\circ\Phi)\,
\frac{\Psi_{ij}}{s_i\circ\Psi_{ij}}\ts\,v\ts.
$$
In particular, the equality $\xi_i\ts(\Phi\ts v)=\Phi\ts v$ holds
if and only if
$$
(s_i\circ\Phi)\,\Psi_{ij}=\Phi\,(s_i\circ\Psi_{ij})\ts. 
$$
But it follows from \eqref{newformula} that $\Psi_{ij}$ and
$s_i\circ\Psi_{ij}$ are mutually prime as polynomials on $\t^*\ts$.
Hence the last displayed
equality holds if and only if $\Phi$ is divisible by $\Psi_{ij}$ and
the ratio $\Phi\ts/\ts\Psi_{ij}$ 
is invariant under the shifted action of $s_i\ts$. 
Thus we get 

\begin{proposition} 
An element $z\in Z^{\ts0}=\U(\t)\ts V^{\ts0}$ 
lies in the subspace $Q$ 
if and only if for each $i=1\lcd r$ one can write $z$ 
as a finite sum of products of the form
$\Theta\,\Psi_{ij}\ts v$ where $v\in V^{\ts0}\cap V_{ij}\ts$,
$\Theta\in\U(\t)$ and $s_i\circ\Theta=\Theta\ts$.
\end{proposition}

%=============================================================================

\section{Symmetric algebra calculation}

Let $\S(\g)$ be the symmetric algebra of $\g\ts$.
Consider the filtration on the $\U(\g)\ts$-bimodule $M=\U(\g)\ot V$ 
arising from the natural filtration on the algebra $\U(\g)$. Let
$\gr M=\S(\g)\ot V$ be the associated graded $\S(\g)\ts$-module. 
%Identifying any vector $v\in V$ with the element $1\ot v$ of $\gr M$, 
%one can write
%$$
%\gr M=\S(\g)V.
%$$
The action of the group $\Gr$ on $M$ induces an action of $\Gr$ 
on $\gr M$, which  coincides with the %natural 
diagonal action of $\Gr$ on 
$\S(\g)\ot V$. The graded subspace of $\gr M$ associated to
$M^{\ts\Gr}$ is $(\gr M)^{\ts\Gr}$.

The graded space associated to $Z=M/(\nminus M+M\n)$ is
$$
\gr Z=\gr M/((\nminus +\n)\gr M)=
(\ts\S(\g)/(\nminus+\n)\ts\S(\g))\ot V.
$$
This space inherits a structure of an $\S(\t)\ts$-module. 
By identifying the quotient vector space
$\S(\g)/(\nminus+\n)\ts\S(\g)$ with $\S(\t)\ts$,
we can identify $\gr Z$ with $\S(\t)\ot V\ts$.
Then $\gr Z^{\ts0}$ gets identified with $\S(\t)\ot V^{\ts0}\ts$.
%$$
%Moreover, the linear map
%$$
%\S(\t)\ot V\to\gr Z,\quad \Phi\ot v\mapsto \Phi v,
%$$
%is a linear isomorphism. 

The map $\ga:M^{\ts\Gr}\to Z$ defined in Section 3 
induces a map $\gr\ga:(\gr M)^{\ts\Gr}\to\gr Z$, 
which is nothing but the restriction to $(\gr M)^{\ts\Gr}$ 
of the canonical projection $\gr M\to\gr Z\ts$.
Note that, when passing to the graded objects, 
the shifted action of the Weyl group $W$ on $\U(\t)$ 
becomes its usual action on $\S(\t)\ts$. Obviously, 
the image of $\gr\ga$ lies in $(\gr Z^{\ts0}){}^W$. 
The next theorem describes this image precisely.

\begin{theorem}
The map $\gr\ga$ is injective. Its image consists of all 
$W$-invariant elements 
$F\in\gr Z^{\ts0}=\S(\t)\ot V^{\ts0}$ 
such that for each $i=1\lcd r$ and
$j=0,1,2,\ldots$ the projection of $F$ to\/
$\S(\t)\ot(V^{\ts0}\cap V_{ij})$ is divisible by $H_i^j$ in 
the first tensor factor.
\end{theorem}

\begin{proof}
By identifying $\g^*$ with $\g$ using a $\Gr\ts$-invariant inner product 
on $\g$, we can regard the elements of $\gr M$ as morphisms 
(polynomial maps) from $\g$ to $V$. Then $(\gr M){}^{\ts\Gr}$ 
is identified with the space of $\Gr\ts$-equivariant morphisms 
from $\g$ to $V$, and the map
$\gr\ga$ is interpreted as the restriction 
of these morphisms to $\t\ts$.

Since generic $\Gr\ts$-orbits in $\g$ intersect $\t\ts$, 
a $\Gr\ts$-equivariant morphism $F:\g\to V$ is uniquely 
determined by its restriction to $\t$. This shows that 
the map $\gr\ga$ is injective.

Now let $\f:\t\to V^{\ts0}$ be a $W$-equivariant morphism. 
Let us try to extend it to a $\Gr\ts$-equivariant morphism 
from $\g$ to $V$. To this end, we will first show that $\f$ 
can be extended to the subset $\g_{\rm\ts sr}$ 
of semisimple regular elements of $\g\ts$.

%For any element $X\in\g$ denote by $X_{\rm\ts s}$ its semisimple part. 
Consider the complement
$D=\g\ts\backslash\ts\g_{\rm\ts sr}\ts$.
This is the set of the elements $X\in\g$
such that the semisimple part %$X_{\rm\ts s}$
of $X$ is not regular, 
that is the multiplicity of the zero root of the characteristic
polynomial $\det\ts(\ts t\cdot 1-\ad\ts X)$ of $\ad\ts X$
is bigger than $\dim\t\ts$. Hence $D$ is the divisor
defined by the equation $P(X)=0\ts$, where $P(X)$
is the coefficient of $t^{\,\dim\t}$ in
the characteristic polynomial of $\ad\ts X\ts$.
In particular, $\g_{\ts\rm sr}$ is Zariski open in $\g\ts$.

The geometry of the action of $\Gr$ on $\g_{\ts\rm sr}$ 
is described as follows. Let $\t_{\ts\rm reg}=\t\cap\g_{\ts\rm sr}$ 
be the set of regular elements of $\t\ts $. 
Consider the homogeneous fibering 
\begin{equation}
\label{fibering}
\Gr*_{\,\Nr(\Tr)}\,\t_{\ts\rm reg}=(\Gr\times\t_{\ts\rm reg})/\Nr(\Tr)
\end{equation}
where the action of any $n\in\Nr(\Tr)$ on $\Gr\times\t_{\ts\rm reg}$ 
is defined by
$$
n:\,(g,H)\mapsto(g\,n^{-1},(\Ad\ts n)\ts H\ts)\ts.
$$
The quotient in (9) is geometric, 
see for instance \cite[Section 4.8]{PV}.
Denote by $\langle g,H\rangle$ the element of %the quotient
\eqref{fibering}
with a representative $(g,H)\in\Gr\times\t_{\ts\rm reg}\ts$.
The action of $\Gr$ on $\Gr\times\t_{\ts\rm reg}$ defined~by 
$$
g':\,(g,H)\mapsto(g'g,H)
$$ 
commutes with the action of $\Nr(\Tr)$ and
induces an action of $\Gr$ on the quotient \eqref{fibering}. 
We have a $\Gr\ts$-equivariant morphism
$$
p:\,
\Gr*_{\,\Nr(\Tr)}\,\t_{\ts\rm reg}\to\g_{\ts\rm sr}:\, 
\langle g,H\rangle\mapsto(\Ad\ts g)\ts H\ts.
$$
Since any regular semisimple element of $\g$ is conjugate by $\Gr$ 
to an element of $\t_{\ts\rm reg}\ts$, and two elements of 
$\t_{\ts\rm reg}$ are $\Gr\ts$-conjugate if and only if they are
$W$-conjugate, $p$ is an isomorphism of algebraic varieties.

Since the morphism
$$
\Gr\times\t_{\ts\rm reg}\to V:\,
(g,H)\mapsto g\,\f(H)
$$
sends each $\Nr(\Tr)$-orbit to one point, it factors through a morphism
$$
\Gr*_{\,\Nr(\Tr)}\,\t_{\ts\rm reg}\to V:\,
\langle g,H\rangle\mapsto g\,\f(H)\ts.
$$
Hence there is a morphism
$\tilde{F}:\g_{\ts\rm sr}\to V$
such that
$$
\tilde{F}((\Ad\ts g)\ts H)=g\,\f(H)
\quad\text{for}\quad 
g\in G
\quad\text{and}\quad 
H\in\t_{\ts\rm reg}.
$$
Clearly, $\tilde{F}$ is $\Gr\ts$-equivariant 
and coincides with $\f$ on $\t_{\ts\rm reg}\ts$.

This means that $\f$ can be extended to a $\Gr\ts$-equivariant rational map
$F:\g\to V$ which is regular (and coincides with $\tilde{F}$) 
on $\g_{\ts\rm sr}\ts$. The question is whether $F$ is regular 
on the whole of $\g\ts$. To find out this, we should study the 
behaviour of $F$ on the irreducible components of the divisor $D\ts$.

The divisor $D$ can be described as follows. 
For any $H\in\t$ denote by $\g_H$ the set of elements of $\g$ 
whose semisimple part is conjugate to $H$. 
Then $\g_H$ is a fiber of the categorical quotient map
% $\g\to\g\,/\hspace{-2pt}/\Gr\ts$
% $=Spec\CC[\g]^G$ 
$\g\to\g\,/\Gr\ts$.
Hence $\g_H$ is an irreducible closed
subvariety of codimension $\dim\t\,$;
see \cite[Section 4.4]{PV}. 
It is the closure of the only orbit open in it, 
namely of the orbit of
$H+U$ where $U$ is a regular nilpotent element of the centralizer of $H\ts$.
It follows that the irreducible
components of $D$ are
$$
D_{\al}=\bigcup_{\substack{H\in\t \\ \al(H)=0}}\,\g_H
$$
where $\al$ runs over representatives of $W$-orbits in $\Delta\ts$. Note that these representatives can be chosen in the set
$\{\al_1\lcd\al_r\}\ts$.

Let $\al_i$ be any simple root and $D_i=D_{\al_i}\ts$. Let
$C_i$ be the set
$$
\{\ts H-2E_i
\ |\ 
H\in\t,\ 
\al_i(H)=0,\ 
\al(H)\neq 0 \ 
\text{for}\ 
\al\in\De\setminus\{\al_i\com-\al_i\}\ts\}\ts.
$$
The subset 
$
(\Ad\Gr)\,C_i\subset\g
$
is Zariski open in the divisor $D_i\ts$. 
If the rational map $F$ has a pole at $D_i\ts$, 
then there exists a point $H-2E_i\in C_i$ such that $F(X)$ 
tends to infinity every time when $X\in\g_{\ts\rm sr}$ tends to
$H-2E_i$. Take such a point and consider the curve
$$
X(t)=(\Ad\ts(\exp t^{-1}E_i))(H+tH_i)=H+tH_i-2E_i.
$$
Here $X(t)\in\g_{\ts\rm sr}$ for $t\neq 0\ts$, and $X(t)$ 
tends to $X(0)=H-2E_i$ when $t$ tends~to $0\ts$. We have
$$
F(X(t))=(\exp t^{-1}E_i)\,\f(H+tH_i)\ts.
$$
For the projection $F_{ij}(X(t))$ of $F(X(t))$ 
to the $\g_i\ts$-isotypic component $V_{ij}$ of $V$, we have
$$
F_{ij}(X(t))
=(\exp t^{-1}E_i)\,\f_{ij}(H+tH_i)
=\sum_{k=0}^j{(\ts t^{k}\ts {k!}\ts)^{-1}\,E_i^k\,\f_{ij}(H+tH_i)}\ts,
$$
where $\f_{ij}(H+tH_i)$ stands for the projection of 
$\f(H+tH_i)$ to $V_{ij}\ts$. Thus $F_{ij}(X(t))$ does not tend 
to infinity if and only if $\f_{ij}(H+tH_i)$ 
is divisible by $t^{\ts j}\ts$. 
Using our identification of $\t^*$ with $\t\ts$,
the latter condition means that $\f_{ij}$ is divisible by $H_i^j\ts$.
\qed
\end{proof}

\begin{remark}
As the proof shows, 
it suffices to check the divisibility condition 
of Theorem~2 for simple roots representing all $W$-orbits in
$\Delta\ts$. In particular, if $\g$ is simple with all roots of 
the same length, then it suffices to check this condition 
for one (arbitrary) simple root.
\end{remark}

\begin{remark}
\label{R3} 
A result of Broer \cite{B}
(in its elegant reformulation due to Panyushev \cite{P})
says that the image of $\gr\ga$ coincides with the space
of $W$-invariant elements of $\S(\t)\otimes V^{\ts0}$ 
if and only if $2\al$ is not a weight of $V$ for any $\al\in\De\ts$. 
%In our notation, 
The last condition implies that 
$V^{\ts0}\cap V_{ij}=\{0\}$
for $i=1\lcd r$ and $j\geqslant2\ts$. If
this holds, then the condition of Theorem~2 on
$F\in\S(\t)\ot V^{\ts0}$ also holds.
Indeed, for $j=1$ the divisibility
of the projection of $F$ to $\S(\t)\ot(V^{\ts0}\cap V_{ij})$
by $H_i$ follows from the $W$-invariance of $F\ts$.
So Theorem 2 gives the sufficiency of Broer's condition.
It would be instructive to deduce 
the necessity of Broer's condition
from Theorem~2 as well.
\end{remark}

%=============================================================================

\section{Proof of  Theorem~1}

The injectivity of $\ga$ immediately follows from the injectivity 
of $\gr\ga$. To prove that the image of $\ga$ is the whole 
of the space $Q$,
it suffices to prove that the image of $\gr\ga$ contains 
(and then coincides with) $\gr Q$.

In the notation of Proposition~4, the leading term of
$\Psi_{ij}$ is $H_i^j$, while the leading term of $\Theta$ 
is invariant under $s_i\ts$. Thus if $z\in Q$, then
the leading term of $z$ is invariant under the Weyl group $W$ 
and its projection to $\S(\t)\ot(V^{\ts0}\cap V_{ij})$ is divisible 
in the first tensor factor by $H_i^j$
for $i=1\lcd r$ and $j=0,1,2,\ldots\,\,$. 
By Theorem~2 this means that the leading term of 
$z$ lies in the image of $\gr\ga$. We get Theorem~1.

\begin{remark}
There is a version of Theorem~1
%which holds 
for disconnected groups. 
Suppose that the action of $\Gr$
on $V$ extends to a linear action of 
an algebraic group $\widehat{\Gr}$ containing
$\Gr$ as the connected component. 
Then the action of $\Gr$ on $M$ naturally extends
to an action of $\widehat{\Gr}$ 
leaving~$M^{\ts\Gr}$ invariant, 
and 
$$
M^{\ts\widehat{\Gr}}=
(M^{\ts\Gr}){}^{\ts\widehat{\Gr}/\hspace{-1pt}\Gr}\ts.
$$

First let $\Gr$ be semisimple.
The adjoint action defines a homomorphism
$\widehat{\Gr}\to\rm{Aut}\,\g$ to the group of
automorphisms of $\g\ts$. It is well known that
$\rm{Aut}\,\g$ is a semidirect product of the group of inner automorphisms
%of $\g$ 
by the (finite) group of diagram automorphisms. 
%which is isomorphic to the symmetry group of the Dynkin diagram. 
Denote by $S$ the inverse image of 
the latter group in $\widehat{\Gr}\ts$,
so that $\widehat{\Gr}=S\Gr\ts$. 
Then $S$ normalizes $\Tr\ts$. By definition, $S$ permutes 
$E_1\lcd E_r$ and (in the same way) permutes $F_1\lcd F_r\ts$.
Therefore $S$ leaves the subspace $\nminus M+M\n\subset M$ 
invariant, and thus acts on $Z$.
Besides, $S$ preserves the set of denominators \eqref{M2}, 
which allows us to extend the action of $S$ to $\Mb$ and to $\Zb\ts$.
%%% THE THREE LINES BELOW WERE ADDED ON 9.10.2009 BY MAXIM %%%
Moreover, the elements $\sih_1\lcd\sih_r\in\Nr(\Tr)$
can be chosen so that $S$ permutes them
in the same way as it does with $E_1\lcd E_r\ts$.
Then $S$ also permutes the Zhelobenko operators. It also leaves the 
subspace $Z^{\ts0}\subset Z$ invariant.
Hence $S$ also leaves the subspace $Q\subset Z$ invariant. 
The map $\ga:M^{\ts\Gr}\to Z$
is $S$-equivariant. Thus the image of 
$M^{\ts\widehat{\Gr}}=(M^{\ts\Gr}){}^{\ts S}$ 
under this map is $Q^{\ts S}\ts$,
the subspace of $S$-invariant elements of $Q\ts$.

Now let $\Gr$ be reductive. Then consider
the natural homomorphism $\widehat{\Gr}\to\rm{Aut}\,\gp$
to the group of automorphisms of the semisimple
part $\gp$ of $\g\ts$.
Denote by $R$ the inverse image in $\widehat{\Gr}$ 
of the group of diagram automorphisms of $\gp\ts$. We have
$\widehat{\Gr}=R\Gr\ts$. Just as above, the action of
$R$ extends to $\Mb$ and $\Zb$ so that
the image of $M^{\ts\widehat{\Gr}}=(M^{\ts\Gr})^{R}$ 
under the map $\ga$ equals $Q^{\ts R}\ts$. 
The intersection $R\cap\Gr$ coincides with the center
of $\Gr$ and acts trivially on $Z^{\ts0}$ and on $Q\ts$.  
Thus we have an action of the finite group $S=R\ts/(R\cap\Gr)$
on $Z^{\ts0}$ and on $Q\ts$,
so that $Q^{\ts R}=Q^{\ts S}\ts$.
Hence the image of $M^{\ts\widehat{\Gr}}$ under the projection $\ga$
equals $Q^{\ts S}\ts$.
\end{remark}

%=============================================================================

\section*{Acknowledgments}

We are grateful to
O.\,Ogievetsky, D.\,Panyushev and A.\,Rosly for useful discussions,
and to the anonymous referee for helpful remarks.
The first author was supported by the RFBR grant 08-01-00667,
joint grant 09-02-90493, and the grant for Support of Scientific Schools 
3036-2008-2.
The second author was supported by the EPSRC grant C511166.
The third author was supported by the RFBR grant 09-01-00648.

%=============================================================================

%=============================================================================

\end{document}